\documentclass[a4paper,10pt, twoside]{amsart}
\usepackage[utf8]{inputenc}
\usepackage[inner=3cm,outer=3cm,top=3cm,headsep=1cm, footskip=1cm,bottom=3cm]{geometry}
\usepackage[english]{babel}
\usepackage{amsfonts,latexsym,rawfonts,amsmath,amssymb,amsthm,amscd}
\usepackage{fancyhdr}
\usepackage{enumerate}
\usepackage{stackrel}
\usepackage{hyperref}
\usepackage{mdwlist}
\usepackage{array} 
\input xy
\xyoption{all}

\hypersetup{
    colorlinks,%
    citecolor=black,%
    filecolor=black,%
    linkcolor=black,%
    urlcolor=black
}

\newtheorem{prop}{\normalfont\bfseries Proposition}[section]
\newtheorem{teo}[prop]{\normalfont\bfseries Theorem}
\newtheorem{lem}[prop]{\normalfont\bfseries Lemma}
\newtheorem{cor}[prop]{\normalfont\bfseries Corollary}

\theoremstyle{definition}
\newtheorem{nada}[prop]{}
\newtheorem{defi}[prop]{\normalfont\bfseries Definition}

\newtheorem{example}[prop]{\normalfont\bfseries Example}

\newcommand{\eps}{\varepsilon}

\newcommand{\inte}{\text{{\scriptsize$\int_0^1$}}}

\def\Aut{\mathrm{Aut}}
\def\AUT{\mathbf{Aut}}

\def\Ho{\mathrm{Ho}}

\def\Dec{\mathrm{Dec}}

\def\Ker{\mathrm{Ker }}
\def\Img{\mathrm{Im }}

\newcommand{\invlim}{\underset{\leftarrow}{\lim}\,}

\newcommand{\pb}{\ar@{}[dr]|{\mbox{\LARGE{$\lrcorner$}}}}

\newcommand{\simr}[1]{\begin{array}{c}\vspace{-.36cm} \simeq\\ \vspace{-.4cm}\text{\tiny{#1}}\vspace{.36cm} \end{array}}

\newcommand{\hto}{\rightsquigarrow}
\newcommand{\xra}[1]{\xrightarrow{#1}}

\newcommand{\dga}[2]{\mathsf{DGA}^{#1}({#2})}
\newcommand{\Fdga}[2]{\mathbf{F}\mathsf{DGA}^{#1}({#2})}

\newcommand{\MHD}{\mathbf{MHD}}

\newcommand{\Sch}[2]{\mathbf{Sch}^{#1}(#2)}

\newcommand{\wt}{\widetilde}
\newcommand{\lra}{\longrightarrow}

\newcommand{\CC}{\mathbb{C}}
\newcommand{\HH}{\mathbb{H}}

\newcommand{\PP}{\mathbb{P}}
\newcommand{\QQ}{\mathbb{Q}}

\newcommand{\ZZ}{\mathbb{Z}}

\newcommand{\Aa}{\mathcal{A}}

\newcommand{\Ee}{\mathcal{E}}

\newcommand{\Oo}{\mathcal{O}}
\newcommand{\Pp}{\mathcal{P}}

\newcommand{\kk}{\mathbf{k}}

\title{$\text{E}_1$-formality of complex algebraic varieties}
\author[J. Cirici] {J. Cirici}
\address[J. Cirici]{
Fachbereich Mathematik und Informatik\\
Freie Universit\"{a}t Berlin\\  Arnimallee 3\\ 
14195 Berlin}
\email{jcirici@math.fu-berlin.de}

\author[F. Guill\'{e}n] {F. Guill\'{e}n}
\address[F. Guill\'{e}n]{Departament
d'\`{A}lgebra i Geometria\\  Universitat de Barcelona\\ Gran Via 585,
08007 Barcelona}
\email{fguillen@ub.edu}
\thanks{Partially supported by the Spanish Ministry of Economy and Competitiveness under project MTM 2009-09557 and by the Generalitat de
Catalunya as members of the team 2009 SGR 119. The first-named author 
wants to acknowledge financial support from the German Research Foundation through the project SFB 647.}


\date{\today}

\begin{document}

\begin{abstract}
Let $X$ be a smooth complex algebraic variety. Morgan \cite{Mo} showed that
the rational homotopy type of $X$ is a formal consequence of 
the differential graded algebra defined by the first term $E_1(X,W)$ of its weight spectral sequence.
In the present work we generalize this result to arbitrary nilpotent complex algebraic varieties (possibly singular and/or non-compact)
and to algebraic morphisms between them. The result for algebraic morphisms generalizes the Formality Theorem of \cite{DGMS}
for compact K\"{a}hler varieties, filling a gap in Morgan's theory concerning functoriality over the rational numbers.
As an application, we study the Hopf invariant of certain algebraic morphisms using intersection theory.
\end{abstract}
\maketitle

\section{Introduction}
Morgan \cite{Mo} introduced \textit{mixed Hodge diagrams of differential graded algebras} (dga's for short) and
proved, using Sullivan's theory of minimal models, the existence of functorial mixed Hodge structures on the rational homotopy groups of
smooth complex algebraic varieties. 
His results were independently extended to the singular case by Hain \cite{Ha} and Navarro \cite{Na}.
Such a mixed Hodge diagram 
is given by a filtered dga $(A_\QQ,W)$ defined over the field $\QQ$ of rational numbers, a bifiltered
dga $(A_\CC,W,F)$ defined over the field $\CC$ of complex numbers, together with a finite string of
filtered quasi-isomorphisms
$(A_\QQ,W)\otimes\CC\longleftrightarrow (A_\CC,W)$ over $\CC$,
in such a way that the cohomology $H(A_\QQ)$ is a graded mixed Hodge structure.
We denote by $\MHD$ the
category whose objects are mixed Hodge diagrams and whose morphisms
are given by level-wise filtered morphisms
that make the corresponding diagrams commute. This differs from Morgan's original definition,
in which level-wise morphisms commute only up to a filtered homotopy.
\\

In the context of sheaf cohomology of dga's, Navarro \cite{Na} introduced the \textit{Thom-Whitney simple}
and used this construction to establish the functoriality of mixed Hodge diagrams associated with
complex algebraic varieties. He defined a functor
$\HH dg:\Sch{}{\CC}\to\Ho(\MHD)$
from the category of complex reduced schemes,
that are separated and of finite type, to the homotopy category of mixed Hodge diagrams (defined by inverting level-wise quasi-isomorphisms),
in such a way that the rational component of $\HH dg(X)$
is the Sullivan-De Rham functor of $X$.
\\

To study the homotopy category $\Ho(\MHD)$ we
 introduce a notion of \textit{minimal object} in the category of mixed Hodge diagrams
and prove the existence of enough models of such type,
adapting the classical construction of Sullivan's minimal models of dga's.
In conjunction with Navarro's functorial construction of mixed Hodge diagrams, this
provides an alternative proof of Morgan's result on the existence of 
functorial mixed Hodge structures in rational homotopy.
A main difference with respect to Morgan's approach is that
our models are objects of a well defined category.
The complex component of our minimal model 
coincides with Morgan's bigraded model (see $\S$6 of \cite{Mo}).
However, we preserve the rational information,
allowing functorial results over the rational numbers. 
In a future paper, we will provide a more detailed study of the homotopy category of mixed Hodge diagrams,
and interpret the existence of minimal models
as a multiplicative version of Beilinson's Theorem on mixed Hodge complexes (see Theorem 2.3 of \cite{Be}).
Using Deligne's splitting of mixed Hodge structures on the minimal models, we prove that
morphisms of nilpotent complex algebraic varieties are $E_1$-formal at the rational level:
the rational homotopy type
is entirely determined by
the first term of the spectral sequence associated with the multiplicative weight filtration.
This generalizes the Formality Theorem of \cite{DGMS}
for compact K\"{a}hler manifolds and
a result due to Morgan (see Theorem 10.1 of \cite{Mo}) for smooth open varieties.
The results agree with Grothendieck’s yoga of weights and can be
viewed as a materialization of his principle in rational homotopy. Indeed,
the weight filtration expresses the way in which the cohomology of the variety is
related to cohomologies of smooth projective varieties. In particular,
$E_1$-formality implies that, at the rational level, complex algebraic varieties have finite-dimensional
models determined by cohomologies of smooth projective varieties. 
\\

This paper is organized as follows.
Section 2 is devoted to the homotopy theory of filtered differential graded commutative algebras. We introduce the notion of $E_r$-formality 
and study its descent properties with respect to field extensions.
In Section 3, we study the homotopy theory of
mixed Hodge diagrams. The existence of minimal models is proven in Theorem $\ref{minim_ho}$ for objects,
and in Theorem $\ref{modelmorfisme}$ for morphisms.
In Section 4, we recall Navarro's construction of mixed Hodge diagrams associated with
complex algebraic varieties. Together with the results of the previous section this leads to the main result of this paper (Theorem $\ref{formalitat_vars}$)
on the $E_1$-formality of complex algebraic varieties.
Lastly, Section 5 is devoted to an application: we study the Hopf invariant of certain algebraic morphisms 
via the weight spectral sequence and intersection theory.

\section{Homotopy Theory of Filtered Algebras}
The category of filtered differential graded commutative algebras over a field $\kk$ of characteristic 0 does not admit a Quillen model structure.
However, the existence of filtered minimal models allows to define a homotopy theory in a non-axiomatic conceptual framework, as done by
Halperin-Tanré \cite{HT}.
In this section we introduce $E_r$-cofibrant filtered dga's and show that 
these satisfy a homotopy lifting property with respect to $E_r$-quasi-isomorphisms. 
We introduce the notions of $E_r$-formality and $r$-splitting and
study their descent properties with respect to field extensions.

\subsection*{Filtered differential graded commutative algebras}
The notion of filtered dga arises from the
compatible combination of a filtered complex with the multiplicative structure of a dga.
For the basic definitions and results on the homotopy theory of dga's we refer to \cite{BG} and \cite{FHT}.
All dga's considered will be non-negatively graded and defined over a field $\kk$ of characteristic 0.
\\

Denote by $\Fdga{}{\kk}$ the category of filtered dga's over $\kk$.
The base field $\mathbf k$ is considered as a filtered
dga with the trivial filtration and the unit map $\eta:\kk\to A$ is filtered.
We will restrict to filtered dga's $(A,W)$ whose filtration is regular and exhaustive:
for each $n\geq 0$ there exists $q\in\ZZ$ such that $W_qA^n=0$, and $A=\cup_pW_pA$.\\

The spectral sequence associated with a filtered dga $A$ is compatible with the multiplicative structure. 
Hence for all $r\geq 0$, the term $E_r^{*,*}(A)$
is a bigraded dga with differential $d_r$ of bidegree $(r,1-r)$.\\

For the rest of this section we fix an integer $r\geq 0$. We adopt the following definition of \cite{HT}.

\begin{defi}
A morphism of filtered dga's $f:A\to B$ is called \textit{$E_r$-quasi-isomorphism} if
 $E_r(f):E_r(A)\to E_r(B)$ is a quasi-isomorphism (the morphism $E_{r+1}(f)$ is an isomorphism).
\end{defi}

Since filtrations are regular and exhaustive, every $E_r$-quasi-isomorphism is a quasi-isomorphism. 
Denote by $\Ee_r$ the class of $E_r$-quasi-isomorphisms, and by
$$\Ho_r(\Fdga{}{\kk}):=\Fdga{}{\kk}[\Ee_r^{-1}]$$
the corresponding localized category. This is the main object of study in the homotopy theory of filtered dga's.
Objects in this category are called \textit{$E_r$-homotopy types}. We have functors
$$\Ho_r(\Fdga{}{\kk})\stackrel{E_r}{\lra} \Ho_0(\Fdga{}{\kk})\stackrel{H}{\lra}\dga{}{\kk}.$$
Deligne's décalage functor of filtered complexes (Definition 1.3.3 of \cite{DeHII}) is compatible with multiplicative structures. 
Therefore it defines a functor
$\Dec:\Fdga{}{\kk}\lra \Fdga{}{\kk}$
which is the identity on morphisms. It follows from Proposition 1.3.4 of loc. cit.
that  $\Ee_{r+1}=\Dec^{-1}(\Ee_r)$. Hence there is an induced functor
$$\Dec:\Ho_{r+1}(\Fdga{}{\kk})\lra \Ho_{r}(\Fdga{}{\kk}).$$
In a subsequent paper we will show that this is in fact an equivalence of categories. In particular,
the study of $E_r$-homotopy types reduces to the case $r=0$.
\begin{defi}Let $(V,W)$ non-negatively graded $\kk$-vector space with a regular and exhaustive filtration.
The \textit{free filtered graded algebra} $\Lambda(V,W)$ defined by $(V,W)$ is the free graded algebra $\Lambda (V)$ 
endowed with the multiplicative filtration induced by the filtration of
$V$. If $A$ has a differential compatible with its multiplicative filtration, then it is called a \textit{free filtered dga}.
\end{defi}

We introduce a notion of homotopy between morphisms suitable to the study of 
$E_r$-homotopy types of filtered dga's.
\begin{nada}
Let $\Lambda(t,dt)$ be the free dga with generators $t$ and $dt$ of degree $0$ and $1$ respectively.
For $r\geq 0$, define an increasing filtration $\sigma[r]$ on $\Lambda(t,dt)$ by letting
$t$ be of pure weight $0$ and $dt$ of pure weight $-r$ and 
extending multiplicatively. Note that $\sigma[0]$ is the trivial filtration, and $\sigma[1]$ is the b\^{e}te filtration.
\end{nada}
\begin{defi}\label{rpath}
The \textit{$r$-path} $P_r(A)$ of a filtered dga $A$ is the dga $A\otimes\Lambda(t,dt)$ with the filtration defined by the convolution of
$W$ and $\sigma[r]$. We have:
$$W_pP_r(A)=\sum_q W_{p-q}A\otimes \sigma[r]_q\Lambda(t,dt)=(W_{p}A\otimes\Lambda(t))\oplus (W_{p+r}A\otimes\Lambda(t)dt).$$
For each $\lambda\in\kk$ there is a map of evaluation of forms $\delta^\lambda:P_r(A)\to A$ defined by $t\mapsto \lambda$ and $dt\mapsto 0$.
\end{defi}
The following Lemma is a matter of verification.
\begin{lem}\label{rpathcommutes}
Let $A$ be a filtered dga.
There are canonical isomorphisms
$$E_r(P_r(A))\cong E_r(A)\otimes\Lambda(t,dt),\quad
\Dec (P_{r+1}(A))\cong P_r(\Dec A).$$
\end{lem}

\begin{defi}\label{rhomotopy}
Let $f,g:A\to B$ be morphisms of filtered dga's. An \textit{$r$-homotopy} from $f$ to $g$ is a
morphism of filtered dga's $h:A\to P_r(B)$ satisfying $\delta^0h=f$ and $\delta^1h=g$.
We denote $h:f\simr{r}g$.
\end{defi}

\begin{lem} If $f\simr{r}g$ then $f=g$ in $\Ho_r(\Fdga{}{\kk})$ and $E_{r+1}(f)=E_{r+1}(g)$.

\end{lem}
\begin{proof}
Since the inclusion $\iota:A\to A\otimes\Lambda(t,dt)$ is a quasi-isomorphism
for any given dga $A$, by Lemma $\ref{rpathcommutes}$ the map $\iota:B\to P_r(B)$ is an $E_r$-quasi-isomorphism.
Hence $E_{r+1}(\delta^0)=E_{r+1}(\delta^1)$.
\end{proof}

\subsection*{Cofibrant filtered algebras}
Cofibrant objects play a key role in homotopy theory. 
We introduce $E_r$-cofibrant dga's
as an adaptation to the filtered setting of the classical notion of Sullivan dga.
The following is a simplified variant of the notion of $(R,r)$-extension introduced by \cite{HT}.
\begin{defi}
Let $A$ be a filtered dga. An \textit{$E_r$-cofibrant extension
of $A$ of degree $n\geq 0$ and weight $p\in\ZZ$} is a filtered dga $A\otimes_\xi\Lambda V$, 
where $V$ is a filtered graded module concentrated in pure degree $n$ and pure weight $p$ and
$\xi:V\to W_{p-r}A$ is a linear map of degree 1 such that $d\circ\xi=0$.
The differential and the filtration on $A\otimes_\xi\Lambda V$ are defined by multiplicative extension.
\end{defi}

\begin{defi}
An \textit{$E_r$-cofibrant dga over $\kk$} is a filtered dga defined by the
colimit of a sequence of $E_r$-cofibrant extensions starting from the base field $\kk$.
\end{defi}

\begin{lem}\label{propietats_rcofs}
Let $A$ be an $E_r$-cofibrant dga. Then:
\begin{enumerate}[(1)]
\item  $A=\Lambda(V,W)$ is a free filtered dga and $d(W_pA)\subset W_{p-r}A$ for all $p\in\ZZ$.
\item As bigraded vector spaces, $E_0(A)=\cdots=E_{r-1}(A)=E_r(A)$.
\end{enumerate}
\end{lem}
\begin{proof}
Assertion (1) follows directly from the definition. 
From (1), the induced differentials of the associated spectral sequence satisfy
$d_0=d_1=\cdots=d_{r-1}=0$. Hence (2) follows.
\end{proof}

\begin{lem}\label{dec_cof} Let $A$ be an $E_{r+1}$-cofibrant dga, with $r\geq 0$. Then:
\begin{enumerate}[(1)]
\item For all $n\geq 0$ and all $p\in \ZZ$, $\Dec W_pA^n=W_{p-n}A^n$.
\item The filtered dga $\Dec A$ is  $E_r$-cofibrant.
\end{enumerate}

\end{lem}
\begin{proof}
By Lemma $\ref{propietats_rcofs}.(1)$ we have $d(W_pA)\subset W_{p-1}A$. Hence (1) follows. To prove (2) it
suffices to note that if $A\otimes_\xi\Lambda(V)$ is an $E_{r+1}$-cofibrant extension of degree $n$ and weight $p$ of $A$ then
$\Dec A\otimes_\xi\Lambda(\Dec V)$ is an $E_{r}$-cofibrant extension of degree $n$ and weight $p-n$ of $\Dec A$. 
\end{proof}

We next show that $E_r$-cofibrant dga's are cofibrant in the sense of \cite{GNPR}.
\begin{teo}\label{rcofs}
Let $M$ be an $E_r$-cofibrant dga. For any solid diagram
 $$
 \xymatrix{
 &A\ar[d]^w_{\wr}\\
 M\ar@{.>}[ur]^g\ar[r]_f&B
 }
 $$
in which $w$ is an $E_r$-quasi-isomorphism there exists a lifting $g$
 together with an $r$-homotopy $h:wg\simr{r}f$. The morphism $g$ is uniquely defined up to $r$-homotopy.
\end{teo}
\begin{proof}To prove the existence of $g$ and $h$ we use induction over $r\geq 0$.
The case $r=0$ is an adaptation to the filtered setting of the proof of Proposition 10.4 of \cite{GM}.
We shall only indicate the main changes. Assume that $w$ is an $E_0$-quasi-isomorphism.
Let $M=M'\otimes_d\Lambda(V)$ be an $E_0$-cofibrant
extension of degree $n$ and weight $p$
and assume that we have defined $g':M'\to A$ together with a $0$-homotopy 
${h}':wg'\simr{0}fi$, where $i:M'\to M$ denotes the inclusion.\\

Denote by $C(w)$ the mapping cone of $w$, with filtration
$W_pC(w)=W_pA[1]\oplus W_pB$ and differential $d(a,b)=(-da,w(a)+db)$.
For each $v\in V$ define a cocycle
$$\wt\theta(v):=\left(g'(dv),f(v)+\inte h'(dv)\right)\in W_pC(w)^n.$$
The assignation $v\mapsto [\wt\theta(v)]$ defines a map $\theta:V\to H^n(W_{p}C(w))$.
Since $f$ is an $E_0$-quasi-isomorphism and filtrations are regular
we have $H^n(W_pC(w))=0$ for all $p\in\ZZ$.
Therefore $\wt\theta(v)$ must be exact. Hence there exists a linear map
$(a,b):V\to W_pC(w)^{n-1}$ such that $d(a,b)=\wt\theta$.
Define a filtered morphism $g:M\to A$ extending $g'$
and a $0$-homotopy $h:M\to P_0(B)$ extending $h'$ by letting
$$g(v):=a(v)\text{ and }h(v):=\left(f(v)+ \int^t_0h'(dv)+d(b(v)\otimes t)  \right).$$
This ends the case $r=0$.
Let $w$ be an $E_{r+1}$-quasi-isomorphism. By Proposition 1.3.4 of \cite{DeHII} the map $\Dec(w)$ is an $E_{r}$-quasi-isomorphism.
By Lemma $\ref{dec_cof}$ the algebra $\Dec M$
is $E_{r}$-cofibrant.
By induction hypothesis
there exists a morphism $g:\Dec M\to \Dec A$ compatible with $\Dec W$
together with an $r$-homotopy $h:wg\simr{r}f$ with respect to $\Dec W$.
Since $M$ is $E_{r+1}$-cofibrant, by Lemma $\ref{dec_cof}$.(1) we have $\Dec W_pM^n=W_{p-n}M^n$.
It follows that $g$ is compatible with $W$, and that $h$ is an $(r+1)$-homotopy with respect to $W$. This ends the inductive step.
\\

To prove that $g$ is uniquely defined up to an $r$-homotopy it suffices to show that if $f_0,f_1:M\to A$ are such that
$h:wf_0\simr{r}wf_1$, then $f_0\simr{r}f_1$. Define the $r$-double mapping path $\Pp^2_r(w)$ of $w$ via the pull-back diagram
$$\xymatrix{
\ar[d]\Pp^2_r(w)\ar[r]&P_r(B)\ar[d]^{(\delta^0,\delta^1)}\\
A\times A\ar[r]^{w\times w}&B\times B.
}
$$
The map $\overline{w}:P_r(A)\to \Pp^2_r(w)$ induced by $(\delta^0,\delta^1,P_r(w))$ is an $E_r$-quasi-isomorphism.
We have a solid diagram 
$$
 \xymatrix{
 &P_r(A)\ar[d]^{\overline{w}}_{\wr}\\
 M\ar@{.>}[ur]^G\ar[r]_-H&\Pp^2_r(w)
 }
 $$
where $H=(f_0,f_1,h)$.
By the existence of liftings proven above 
there is a morphism $G$ such that $\overline{w}G\simr{r} H$. 
Then $G:f_0\simr{r}f_1$ is an $r$-homotopy from $f_0$ to $f_1$.
\end{proof}

\subsection*{Splittings and Formality}The notion of
$E_r$-formality is a homotopic version of the existence of $r$-splittings, 
and generalizes the classical notion of \cite{Su} and \cite{HS} of formality of dga's,
to the filtered setting.

\begin{defi}\label{def_split}
An \textit{$r$-splitting} of a filtered dga $A$ is a
direct sum decomposition $A=\bigoplus A^{p,q}$ into subspaces $A^{p,q}$ such that for all $p,q\in\ZZ$,
$$d(A^{p,q})\subset A^{p+r,q-r+1},\,A^{p,q}\cdot A^{p',q'}\subset A^{p+p',q+q'}\text{ and }
W_mA^n=\bigoplus_{p\leq m}A^{-p,n+p}.$$
\end{defi}
The $r$th-term of the spectral sequence associated with a filtered dga admits a natural filtration
$$W_pE_r(A):=\bigoplus_{i\leq p}E_r^{-i,*}(A).$$
Hence $(E_r(A),d_r,W)$ is a filtered dga with an $r$-splitting.
The following result is straightforward.
\begin{prop}\label{splitting_ss}
If a filtered dga $(A,d,W)$ admits an $r$-splitting $A=\bigoplus A^{p,q}$ then the differentials of its spectral sequence 
satisfy $d_0=\cdots =d_{r-1}=0$,
and there is an isomorphism of filtered dga's $\pi:(A,d,W)\stackrel{\cong}{\lra} (E_r(A),d_r,W)$,
such that $\pi(A^{-p,n+p})=Gr_p^WA^{n}=E_r^{-p,n+p}(A)$.
\end{prop}

The following example exhibits the relation between $1$-splittings and the classical formality of dga's.
\begin{example}
Let $A$ be a filtered dga, where $W$ is the trivial filtration $0=W_{-1}A\subset W_0A=A$. 
The bigraded model $M\to A$ of Halperin-Stasheff (see 3.4 of \cite{HS}) is $E_1$-cofibrant,
and $A$ is formal (the dga's $(A,d)$ and $(H(A),0)$ have the same Sullivan minimal model)
if and only if $M$ admits a 1-splitting.
\end{example}

\begin{defi}
A filtered dga $(A,d,W)$ is said to be \textit{$E_r$-formal} if there exists an isomorphism
$(A,d,W)\lra (E_r(A),d,W)$ in the homotopy category $\Ho_r(\Fdga{}{\kk})$.
\end{defi}
In particular if a filtered dga is connected by a string of $E_r$-quasi-isomorphisms 
to a dga admitting an $r$-splitting, then it is $E_r$-formal.
\\

The previous definitions are naturally extended to morphisms.
\begin{defi}\label{def_split_mor}
Let $f:A\to B$ be a morphism of filtered dga's. We say that \textit{$f$ admits an $r$-splitting} if
$A$ and $B$ admit $r$-splittings and $f$ is compatible with them.
\end{defi}

\begin{defi}
A morphism of filtered dga's $f:A\to B$ is said to be \textit{$E_r$-formal} if there exists
a commutative diagram 
$$
\xymatrix{
\ar[d]^{f}(A,d,W)\ar[r]^-\cong&(E_1(A),d_1,W)\ar[d]^{E_1(f)}\\
(B,d,W)\ar[r]^-\cong&(E_1(B),d_1,W)\\
}
$$ in the homotopy category $\Ho_r(\Fdga{}{\kk})$, where the horizontal arrows are isomorphisms.
\end{defi}

\subsection*{Descent of splittings}
The descent of formality of nilpotent dga's from $\CC$ to $\QQ$ is proved in Theorem 12.1 of \cite{Su}.
The proof is based on the fact that the existence of certain grading automorphisms does not depend on the base field.
Following this scheme, we characterize the existence of $r$-splittings 
of finitely generated $E_r$-cofibrant dga's in terms of the existence of liftings of certain $r$-bigrading automorphisms.

\begin{nada}
Let us fix some notations about group
schemes. Given a filtered dga $A$ denote by $\Aut_W(A)$ the set of its filtered automorphisms. Likewise,
denote by $\Aut(E_r(A))$ the set of morphisms of bigraded dga's from $E_r(A)$ to itself.
We have a morphism $E_r:\Aut_W(A)\to \Aut(E_r(A))$.
\\

Let  $\kk\to R$ be a commutative $\kk$-algebra. If $A$ is a dga over $\kk$, its extension of scalars
$A\otimes_{\kk} R$ is a dga over $R$, and the
correspondence
$$R \mapsto \AUT_W{(A)}(R) = \Aut_W(A \otimes_{\mathbf{k}}R)$$
defines a functor
$\AUT_W(A):\mathbf{alg}_\kk \to \mathbf{Gr}$
from the category  $\mathbf{alg}_\kk$  of commutative
$\kk$-algebras, to the category
 $\mathbf{Gr}$ of groups. It is clear that
$
\AUT_W{(A)} (\mathbf{k}) = \Aut_W(A).
$
\end{nada}

\begin{prop}
\label{algebraicos} Let $A$ be a finitely generated $E_r$-cofibrant dga over $\kk$. Then:
\begin{enumerate}[(1)]
\item
$\Aut_W(A)$ is an algebraic matrix group over $\mathbf k$.
\item $\AUT_W(A)$ is an algebraic affine group scheme over $\mathbf k$
represented by 
$\Aut_W(A)$.
\item $E_r$ defines a morphism $\mathbf{E_r} :\AUT_W{(A)} \longrightarrow \AUT(E_r(A))$  of algebraic affine group
schemes.
\item The kernel
$\mathbf{N} := \ker \left( \mathbf{E_r} :\AUT_W(A) \longrightarrow
\AUT(E_r(A)) \right)$
is a unipotent algebraic affine group scheme over $\mathbf k$.
\end{enumerate}
\end{prop}

\begin{proof} 
Since $A$ is finitely generated, for a sufficiently large $N\geq 0$,
 $\Aut_W(A)$ is the closed subgroup of $\mathbf{GL}_N(\kk)$ defined by the
polynomial equations that express compatibility with differentials, products and filtrations.
Thus $\Aut_W(A)$ is an algebraic matrix group. Moreover,
$\AUT_W(A)$ is obviously the algebraic affine group scheme
represented by $\Aut_W(A)$. Hence (1) and (2) are satisfied.
For every commutative $\kk$-algebra $R$, the map
$$
\AUT_W{(A)}(R) = \Aut_W(A \otimes_{\mathbf{k}}R) \longrightarrow \Aut(E_r(A)\otimes_{\mathbf{k}}R) = \AUT(E_r(A))(R)
$$
is a morphism of groups which is natural in $R$. Thus (3)
follows.
Since by (2) both groups
 $\AUT_W(A)$ and $\AUT(E_r(A))$ are
algebraic, and $\kk$ has  zero
characteristic, the kernel $\mathbf{N}$ is represented by an
algebraic matrix group defined over $\kk$ (see \cite{Borel}, Corollary 15.4).
Therefore to prove (4) it suffices to verify that all elements in
$\mathbf{N}(\kk)$ are unipotent.
Given $f \in \mathbf{N}(\mathbf{k})$, consider the multiplicative Jordan decomposition $f=f_s\cdot f_u$ into semi-simple and unipotent parts.
By Theorem 4.4 of \cite{Borel} we have $f_s,f_u\in\AUT_W(A)(\kk)$. 
Since $E_r(f)=1$ and an algebraic group morphism preserves semi-simple and unipotent parts,
it follows that $E_r(f_s)=E_r(f_u)=1$. Let $A_1=\Ker(f_s-I)$ and decompose $A$ into invariant subspaces $A=A_1\oplus B$.
Since $df_s=f_sd$ this decomposition satisfies $d(A_1)\subset A_1$ and $dB\subset B$.
Hence both $A_1$ and $B$ are filtered subcomplexes of $A$
satisfying $d(W_pA_1)\subset W_{p-r}A_1$ and $d(W_pB)\subset W_{p-r}B$. Therefore we have
$$E_r(A)=E_0(A)=E_0(A_1)\oplus E_0(B)=E_r(A_1)\oplus E_r(B).$$
Since $E_r(A)$ contains nothing but the eigenspaces of eigenvalue $1$, we have $E_r(B)=E_0(B)=0$,
and so $B=0$. Therefore $f_s=1$ and $f$ is unipotent.
\end{proof}

\begin{defi}\label{gradingautodef}
Let $\alpha\in\kk^*$ (not a root of unity). The \textit{$r$-bigrading automorphism of $E_r(A)$ associated with $\alpha$} 
is the automorphism $\varphi_\alpha:E_r(A)\to E_r(A)$ defined by
$$\varphi_\alpha(a)=\alpha^{nr+p}a,\text{ for }a\in E_r^{-p,n+p}(A).$$
\end{defi}

\begin{lem}
\label{split_autos}
Let $A$ be a finitely generated $E_r$-cofibrant dga defined over a field $\kk$ of characteristic 0. The following are equivalent:
\begin{enumerate}[(1)]
\item The filtered dga $A$ admits an $r$-splitting.
\item The morphism $E_r:\Aut_W(A)\lra \Aut(E_r(A))$ is surjective.
\item There exists $\alpha\in\kk^*$ (not root of unity) together with an automorphism $\Phi\in \Aut_W(A)$
such that $E_r(\Phi)=\varphi_\alpha$ is the $r$-bigrading automorphism of $E_r(A)$ associated with $\alpha$.
\end{enumerate}
\end{lem}
\begin{proof}
By Proposition $\ref{splitting_ss}$ it follows that (1) implies (2). It is trivial that (2) implies (3).
We show that (3) implies (1). Let $\Phi\in \Aut_W(A)$ be such that $E_r(\Phi)=\varphi_\alpha$.
Consider the multiplicative Jordan decomposition $\Phi=\Phi_s\cdot \Phi_u$.
By Theorem 4.4 of \cite{Borel} we have that $\Phi_s,\Phi_u\in\Aut_W(A)$. 
Since $A$ is finitely generated there is a vector space decomposition of the form $A=A'\oplus B$,
where
$$A'=\bigoplus A^{p,q}\text{ with }A^{-p,n+p}:=\Ker(\Phi_s-\alpha^{nr+p}I)\cap A^n$$
and $B$ is the complementary subspace corresponding to the remaining factors of the characteristic polynomial of $\Phi_s$.
Since $dA^n\subset A^{n+1}$ and $d\Phi_s=\Phi_sd$ this decomposition satisfies
$$d(A^{p,q})\subset A^{p+r,q-r+1}\text{ and }dB\subset B.$$
By analogous reasoning used in the proof of Proposition $\ref{algebraicos}$.(4) one concludes that $B=0$.\\

To show that $W_pA=\bigoplus_{i\leq p}A^{-i,*}$ it suffices to see that
$A^{-p,*}\subset W_pA$. Let $x\in A^{-p,n+p}$, and let $q$ be the smallest integer such that $x\in W_qA$.
Then $x$ defines a class $x+W_{q-1}A\in E_r^{-q,n+q}$, and
$$\varphi_\alpha(x+W_{q-1}A)=\alpha^{nr+q}x+W_{q-1}A=\Phi(x)+W_{q-1}A=\alpha^{nr+p}x+W_{q-1}A.$$
It follows that $(\alpha^{q}-\alpha^{p})\alpha^{nr}x\in W_{q-1}A$. Since $x\notin W_{q-1}A$ we have $q=p$, hence $x\in W_pA$.
Since $\Phi$ is multiplicative
$A^{p,q}\cdot A^{p',q'}\subset A^{p+p',q+q'}$, and
the above decomposition
is an $r$-splitting of $A$.
\end{proof}

Based on the Sullivan formality criterion of Theorem 1 of \cite{FeTa},
a descent of formality for morphisms of dga's is proved in Theorem 3.2 of \cite{Roig}.
We follow the same scheme to characterize the existence of $r$-splittings of morphisms of filtered dga's.

\begin{nada}Given a morphism $f:A\to B$ of filtered dga's, denote by $\Aut_W(f)$ the set of pairs $(F^A,F^B)$, where
$F^A\in \Aut_W(A)$ and $F^B\in\Aut_W(B)$ are such that $f F^A=F^Bf$. The set
$\Aut(E_r(f))$ is defined analogously. We have a natural morphism
$E_r:\Aut_W(f)\lra \Aut(E_r(f)).$
\end{nada}

\begin{prop}
\label{algebraicos_map} Let  $f:A\to B$ be a morphism of finitely generated $E_r$-cofibrant dga's over $\kk$.
\begin{enumerate}[(1)]
\item
$\Aut_W(f)$ is an algebraic matrix group over $\mathbf k$.
\item $\AUT_W(f)$ is an algebraic affine group scheme over $\mathbf k$
represented by 
$\Aut_W(f)$.
\item $E_r$ defines a morphism $\mathbf{E_r} :\AUT_W{(f)} \longrightarrow \AUT(E_r(f))$  of algebraic affine group
schemes.
\item The kernel
$\mathbf{N} := \ker \left( \mathbf{E_r} :\AUT_W(f) \longrightarrow
\AUT(E_r(f)) \right)$
is a unipotent algebraic affine group scheme over $\mathbf k$.
\end{enumerate}
\end{prop}
\begin{proof}
The proof follows analogously to that of Proposition $\ref{algebraicos}$.
\end{proof}

\begin{lem}\label{split_autos_morf}
Let $f:A\to B$ be a morphism of finitely generated $E_r$-cofibrant dga's over $\kk$. The following are equivalent:
\begin{enumerate}[(1)]
\item The morphism $f:A\to B$ admits an $r$-splitting.
\item The morphism $E_r:\Aut_W(f)\lra \Aut(E_r(f))$ is surjective.
\item There exists $\alpha\in\kk^*$ (not root of unity) together $\Phi \in \Aut_W(f)$
such that $E_r(\Phi)=\varphi_\alpha$ is induced by the level-wise $r$-bigrading automorphism associated with $\alpha$.
\end{enumerate}
\end{lem}
\begin{proof}
The proof follows analogously to that of Lemma $\ref{split_autos}$.
\end{proof}

\begin{teo}\label{descens_rsplittings} Let $f:A\to B$ be a morphism of finitely generated  $E_r$-cofibrant dga's 
over $\kk$, and let $\kk\subset\mathbf{K}$ be a field extension.
Then $f$ admits an $r$-splitting if and only if $f_{\mathbf{K}}:=f\otimes_\kk\mathbf{K}$ admits an $r$-splitting.
\end{teo}
\begin{proof}
We may assume that $\mathbf{K}$ is algebraically closed.
If $f_{\mathbf{K}}$ admits an $r$-splitting, 
the map 
$
\AUT_W(f)(\mathbf{K}) \longrightarrow \AUT(E_1(f))(\mathbf{K})
$
is surjective by Lemma $\ref{split_autos}$. From Section 18.1 of \cite{Waterhouse} there is an
exact sequence of groups
$$
1 \longrightarrow \mathbf{N}(\mathbf{\kk}) \longrightarrow
\AUT_W(f)(\kk) \longrightarrow \AUT(E_1(f)) (\kk)
\longrightarrow H^1(\mathbf{K}/\kk, \mathbf{N})
\longrightarrow \dots
$$
where $\mathbf{N}$ is unipotent by Proposition $\ref{algebraicos_map}$.
Since $\mathbf{\kk}$ has characteristic zero the group
$H^1(\mathbf{K}/\kk, \mathbf{N})$ is trivial (see 
Example 18.2.e of \cite{Waterhouse}). 
This gives the exact sequence
$$
1 \rightarrow \mathbf{N}(\mathbf{\kk}) \longrightarrow \Aut_W(f)
\longrightarrow \Aut (E_1(f)) \longrightarrow 1.
$$
In particular the middle arrow is surjective. Hence $f$ admits a $1$-splitting by 
Lemma $\ref{split_autos_morf}$.
 \end{proof}

\section{Homotopy Theory of Mixed Hodge Diagrams}\label{homotopytheory}
We prove the existence of minimal models of mixed Hodge diagrams,
as an adaptation of the classical construction of Sullivan's minimal models.
Using Deligne's splitting of mixed Hodge structures we show that these models admit a 1-splitting,
leading to a result of $E_1$-formality of mixed Hodge diagrams.
\subsection*{Mixed Hodge diagrams}
Throughout this section we let
$I=\{0\to 1\leftarrow 2\to\cdots\leftarrow s\}$ be a finite category of zig-zag type and fixed length $s$.
The following is a multiplicative version of the original notion of mixed Hodge complex (see 8.1 of \cite{DeHIII}).
\begin{defi}
A \textit{mixed Hodge diagram (of dga's over $\QQ$ of type $I$)} consists in:
\begin{enumerate}[(i)]
\item  a filtered dga $(A_{\QQ},W)$ over $\QQ$,
\item  a bifiltered dga $(A_{\CC},W,F)$ over $\CC$, together with
\item an $E_1$-quasi-isomorphism
$\varphi_u:(A_i,W)\to (A_j,W)$ over $\CC$, for each $u:i\to j$ of $I$,
with $A_0=A_\QQ\otimes\CC$ and $A_s=A_\CC$.
\end{enumerate}
In addition, the following axioms are satisfied:
\begin{enumerate}
\item [($\text{MH}_0$)] The weight filtrations $W$ are regular and exhaustive. The Hodge filtration $F$ is biregular.
The cohomology $H(A_\QQ)$ has finite type.
\item [($\text{MH}_1$)] For all $p\in\ZZ$, the differential of $Gr_p^WA_\CC$ is strictly compatible with the filtration $F$.
\item [($\text{MH}_2$)] For all $n\geq 0$ and all $p\in\ZZ$, the filtration $F$ induced on $H^n(Gr^W_pA_{\CC})$ defines a pure Hodge structure of
weight $p+n$ on $H^n(Gr^W_pA_\QQ)$.
\end{enumerate}
Such a diagram is denoted as
$$A=\left((A_\QQ,W)\stackrel{\varphi}{\dashleftarrow\dashrightarrow}(A_\CC,W,F)\right).$$
\end{defi}
Note that axiom ($\text{MH}_2$) implies that for all $n\geq 0$ the triple $(H^n(A_\QQ),\Dec W,F)$ is a mixed Hodge structure over $\QQ$.

\begin{defi}A \textit{pre-morphism of mixed Hodge diagrams $f:A\dashrightarrow B$} consists in:
\begin{enumerate}[(i)]
\item  a morphism of filtered dga's $f_\QQ:(A_\QQ,W)\to (B_\QQ,W)$ over $\QQ$,
\item a morphism of bifiltered dga's $f_\CC:(A_\CC,W,F)\to (B_\CC,W,F)$ over $\CC$, and
\item  a family of morphisms of filtered dga's $f_i:(A_i,W)\to (B_i,W)$ over $\CC$, for each $i\in I$, with
$f_0=f_\QQ\otimes\CC$ and $f_s=f_\CC$.
\end{enumerate}
\end{defi}

\begin{defi}
A pre-morphism $f:K\dashrightarrow L$ is said to be a \textit{quasi-isomorphism} if
$f_\QQ$, $f_\CC$ and $f_i$ are quasi-isomorphisms (the induced morphisms $H^*(f_\QQ)$, $H^*(f_i)$ and $H^*(f_\CC)$ are isomorphisms).
\end{defi}

The following result is an easy consequence of Scholie 8.1.9 of \cite{DeHIII},
stating that the spectral sequences associated with the Hodge and the weight filtrations degenerate
at the stages $E_1$ and $E_2$ respectively.
\begin{lem}\label{quis_es_we}
Let $f:A\dashrightarrow B$ be a quasi-isomorphism of mixed Hodge diagrams. Then
$f_\QQ$ and $f_i$ are $E_1$-quasi-isomorphisms and $f_\CC$ is an $E_{1,0}$-quasi-isomorphism 
(the induced morphisms $E_2(f_\QQ)$, $E_2(f_i)$ for all $i\in I$ and $E_2(Gr^p_Ff_\CC)$ for all $p\in\ZZ$ are isomorphisms).
\end{lem}

\begin{defi}A \textit{morphism of mixed Hodge diagrams} is a pre-morphism $f:A\dashrightarrow B$ such that 
for all $u:i\to j$ of $I$ the diagram
$$
\xymatrix{
(A_i,W)\ar[d]_{f_i}\ar[r]^{\varphi_u^A}&(A_j,W)\ar[d]^{f_j}\\
(B_i,W)\ar[r]^{\varphi_u^B}&(B_j,W)
}
$$
commutes.  We denote such a morphism by $f:A\to B$.
\end{defi}
Denote by $\MHD$ the category of mixed Hodge diagrams over $\QQ$ of a fixed type $I$ and by $\Ho(\MHD)$
the localized category of mixed Hodge diagrams with respect to the class of quasi-isomorphisms.
By Lemma $\ref{quis_es_we}$, the forgetful functor
$U_\QQ:\MHD\lra \Fdga{}{\QQ}$
defined by sending every mixed Hodge diagram $A$ to the filtered dga $(A_\QQ,W)$
sends quasi-isomorphisms
of mixed Hodge diagrams to $E_1$-quasi-isomorphisms of filtered dga's.
Hence there is an induced functor
$$U_\QQ:\Ho\left(\MHD\right)\lra \Ho_1\left(\Fdga{}{\QQ}\right).$$

For the construction of minimal models we shall need a broader class of maps between mixed Hodge diagrams,
defined by level-wise morphisms commuting only up to $1$-homotopy. 

\begin{defi}\label{homorphism}A \textit{ho-morphism of mixed Hodge diagrams} is given by a pre-morphism $f:A\dashrightarrow B$, together with
a family of $1$-homotopies (see Definition $\ref{rhomotopy}$) $F_u:f_j\varphi_u^A\simeq \varphi_u^Bf_i$ for all $u:i\to j$ of $I$,
making the diagram
$$
\xymatrix{
(A_i,W)\ar@{=>}[dr]^{F_u}\ar[d]_{f_i}\ar[r]^{\varphi_u^A}&(A_j,W)\ar[d]^{f_j}\\
(B_i,W)\ar[r]^{\varphi_u^B}&(B_j,W)
}
$$
$1$-homotopy commute. We denote such a ho-morphism by $f:A\hto B$.
\end{defi}
In general, ho-morphisms cannot be composed. Therefore unlike (strictly commutative) morphisms, they do not define a category.
Every morphism is a ho-morphism with trivial homotopies.
In a subsequent paper we will show how homotopy classes of ho-morphisms between minimal objects define a category
equivalent to the homotopy category of mixed Hodge diagrams.  
\\

We introduce the mapping cone of a ho-morphism of mixed Hodge diagrams and show that under the choice of certain filtrations, the mapping cone
is a mixed Hodge complex. 
\begin{defi} Let $f:A\to B$ be a morphism of filtered dga's. The \textit{$r$-cone of $f$} is the filtered complex $C_r(f)$ defined by
$$W_pC_r(f):=W_{p-r}A^{n+1}\oplus W_pB^n,
\text{ with }d=(-da,f(a)+db).$$ 
For a bifiltered morphism $f:A\to B$, the
 \textit{$(r,s)$-cone} $C_{r,s}(f)$ is defined analogously:
$$W_pF^qC_r(f):=W_{p-r}F^{q+s}A^{n+1}\oplus W_pF^qB^n.$$
\end{defi}
\begin{nada}
Let $f:A\hto B$ be a ho-morphism of mixed Hodge diagrams.
For each $u:i\to j$ of $I$,
the 1-homotopy of filtered dga's $F_u:A_i\to P_1(B_j)$ from $f_j\varphi_u^A$ to $\varphi_u^Bf_i$ of the ho-morphism $f$
gives rise to a homotopy 
$$W_p\inte F_u:W_pA_i\lra W_{p+1}B_j[-1]$$
at the level of underlying complexes of vector spaces (see \cite{GM}, X.10.3).
This allows to define filtered morphisms $\varphi_u^f:C_1(f_i)\to C_1(f_j)$ by letting
$$(a,b)\mapsto(\varphi_u^A(a),\varphi_u^B(b)+\inte F_u(a))$$
\end{nada}
\begin{defi}
The \textit{mixed cone} of a ho-morphism $f:A\hto B$ of mixed Hodge diagrams is the diagram of filtered complexes given by
$$C(f)=\left((C_1(f_\QQ),W)\stackrel{\varphi^f}{\dashleftarrow\dashrightarrow}(C_{1,0}(f_\CC),W,F)\right).$$
\end{defi}

\begin{lem}[cf. \cite{PS}, Theorem 3.22]\label{mixedcone}
Let $f:A\hto B$ be a ho-morphism of mixed Hodge diagrams.
The mixed cone of $f$ is a mixed Hodge complex.
\end{lem}
\begin{proof}
Consider the commutative diagram with exact rows
$$
\xymatrix{
0\ar[r]&\ar[d]^{\varphi_u^B}B_i\ar[r]&C(f_i)\ar[d]^{\varphi_u^{f}}\ar[r]&A_i[1]\ar[d]^{\varphi_u^A}\ar[r]&0\\
0\ar[r]&B_j\ar[r]&C(f_j)\ar[r]&A_j[1]\ar[r]&0\\
}
$$
By the five lemma, $\varphi_u^f$ is an $E_1$-quasi-isomorphism. Condition ($\text{MH}_0$) is trivial.
For all $p\in\ZZ$,
$$Gr_p^WC(f_\CC)=Gr_{p-1}^WA_\CC[1]\oplus Gr_p^WB_\CC.$$
Hence at the graded level, the contribution of $f_\CC$ to the differential of $C(f_\CC)$ vanishes.
Therefore we have a direct sum decomposition of complexes compatible with the Hodge filtration $F$,
and ($\text{MH}_1$) and ($\text{MH}_2$) follow.
\end{proof}

\subsection*{Minimal models}
The following technical results will be of use for the construction of minimal models of mixed Hodge diagrams.
Let us first recall Deligne's splitting (\cite{DeHII}, 1.2.11, see also \cite{GS}, Lemma 1.12). This is a global decomposition
for any given mixed Hodge structure, which generalizes the decomposition of pure Hodge structures.
\begin{lem}[\cite{DeHII}, 1.2.11]\label{Desplitting}
Let $(V,W,F)$ be a mixed Hodge structure defined over $\kk$. Then $V_\CC=V\otimes_\kk\CC$ admits
a direct sum decomposition
$V_\CC=\bigoplus_{p,q} I^{p,q}$ such that the filtrations $W$ and $F$ defined on $V_\CC$ are given by
$$W_mV_\CC=\bigoplus_{p+q\leq m} I^{p,q}\text{ and }F^lV_\CC=\bigoplus_{p\geq l} I^{p,q}.$$
The above decomposition is functorial for morphisms of mixed Hodge structures.
\end{lem}

\begin{lem}\label{seccions_compatibles}
Let $A$ be a mixed Hodge diagram.
\begin{enumerate}[(1)]
 \item  There are sections $\sigma^n_\QQ:H^n(A_\QQ)\to Z^n(A_\QQ)$ and $\sigma^n_i:H^n(A_i)\to Z^n(A_i)$ of the projection,
 which are compatible with the weight filtration
$W$.
\item  There exists a section $\sigma^n_\CC:H^n(A_\CC)\to Z^n(A_\CC)$ of the projection, which is compatible with both filtrations $W$ and $F$.
\end{enumerate}
\end{lem}
\begin{proof}
Since the differential of $A_\QQ$ is strictly compatible with the filtration
$\Dec W$, there is a section $\sigma_\QQ:H^n(A_\QQ)\to Z^n(A_\QQ)$ compatible with $\Dec W$.
Since $\Dec W_pH^n(A_\QQ)=W_{p-n}H^n(A_\QQ)$, the map $\sigma_\QQ$ is compatible with $W$. For $\sigma_i:H^n(A_i)\to Z^n(A_i)$
the proof is analogous. This proves (1).
\\

Let us prove (2). Since $(H^n(A_\QQ),\Dec W,F)$ is a mixed Hodge structure, by Lemma $\ref{Desplitting}$
there exists a 
direct sum decomposition $H^n(A_\CC)=\bigoplus I^{p,q}$ with $I^{p,q}\subset \Dec W_{p+q}F^pH^n(A_\CC)$ such that
$$W_mH^n(A_\CC)=\Dec W_{m+n}H^n(A_\CC)=\bigoplus_{p+q\leq m+n} I^{p,q}\text{ and }F^lH^n(A_\CC)=\bigoplus_{p\geq l} I^{p,q}.$$
Therefore it suffices to define sections $\sigma^{p,q}:I^{p,q}\to Z^n(A_\CC)$.
By Scholie 8.1.9 of \cite{DeHIII}
the four spectral sequences
$$
\xymatrix@R=.2pc{
E_1(Gr^{\Dec W}_\bullet A_\CC,F)\ar@{=}[dd]\ar@{=>}[r]&E_1(A_\CC,\Dec W)\ar@{=>}[rd]\\
&&H(K)\\
E_1(Gr_F^\bullet A_\CC,\Dec W)\ar@{=>}[r]&E_1(A_\CC,F)\ar@{=>}[ru]\\
}
$$
degenerate at $E_1$.
It follows that
the induced filtrations in cohomology are given by:
$$\Dec W_pF^qH^n(A_\CC)=\Img\{H^n(\Dec W_pF^qA_\CC)\lra H^n(A_\CC)\}.$$
Since $I^{p,q}\subset \Dec W_{p+q}F^qH^n(A_\CC)$ we have $\sigma^{p,q}(I^{p,q})\subset \Dec W_{p+q}F^pA_\CC$.
Define 
$$\sigma^n_\CC:=\bigoplus \sigma^{p,q}:H^n(A_\CC)\lra A_\CC.$$
For the weight filtration we have:
$$\sigma^n_\CC(W_mH^n(A_\CC))=\bigoplus_{p+q\leq m+n} \sigma^{p,q}(I^{p,q})\subset
\bigoplus_{p+q\leq m+n}\Dec W_{p+q}A_\CC^n
\subset \Dec W_{m+n}A_\CC^n\subset W_mA_\CC^n.$$
Therefore
$\sigma_\CC$ is compatible with $W$.
For the Hodge filtration we have:
$$\sigma^n_\CC(F^lH^n(A_\CC))=\bigoplus_{p\geq l} \sigma^{p,q}(I^{p,q})\subset \sum_{p\geq l} F^{p}A_\CC
\subset F^lA_\CC.$$
Therefore $\sigma_\CC$ is compatible with $F$.

\end{proof}

\begin{defi}
A mixed Hodge diagram $A$ is said to be \textit{0-connected} if the unit map $\eta:\QQ\to A_\QQ$
induces an isomorphism $\QQ\cong H^0(A_\QQ)$.
\end{defi}

\begin{defi}
A \textit{mixed Hodge dga} is a filtered dga $(A,d,W)$ over $\QQ$,
together with a filtration $F$ on $A_\CC:=A\otimes_\QQ \CC$, such that for each $n\geq 0$, the triple
$(A^n,\Dec W,F)$
is a mixed Hodge structure and the differentials $d:A^n\to A^{n+1}$ and products
$A^n\otimes A^m\to A^{n+m}$
are morphisms of mixed Hodge structures.
\end{defi}
The cohomology of every mixed Hodge diagram is a mixed Hodge dga with trivial differential. Conversely,
since the category of mixed Hodge structures is abelian (\cite{DeHII}, Theorem 2.3.5), every
mixed Hodge dga is a mixed Hodge diagram in which the comparison morphisms are identities.
We will show that every 0-connected mixed Hodge diagram is quasi-isomorphic to a mixed Hodge dga satisfying the following minimality condition.

\begin{defi}
Let $A$ be a mixed Hodge dga. A \textit{mixed Hodge extension} of $A$ of degree $n$ is
a mixed Hodge dga $A\otimes_\xi\Lambda(V)$
where $(V,W)$ is a filtered graded module concentrated in pure degree $n$ and
$\xi:V\to A$ is a linear map of degree 1 such that $d\circ\xi=0$ and $\xi(W_pV)\subset W_{p-1}A$.
In addition, the vector space $V\otimes\CC$ has a filtration $F$ compatible with $\xi$
making the triple $(V,\Dec W,F)$ into a mixed Hodge structure.
The differentials and filtrations on $A\otimes_\xi\Lambda(V)$ are defined by multiplicative extension. 
Such an extension is said to be \textit{minimal} if $A$ is augmented and $\xi(V)\subset A^+\cdot A^+$.
\end{defi}

\begin{defi}A mixed Hodge dga is said to be \textit{minimal} if it is the colimit of a sequence of
minimal mixed Hodge extensions starting from the base field $\QQ$
 endowed with the trivial mixed Hodge structure.
\end{defi}
Note that in particular, every minimal mixed Hodge dga is a Sullivan minimal $E_1$-cofibrant dga.
To construct minimal models for 0-connected mixed Hodge diagrams we adapt the classical step by step
construction of Sullivan minimal models for 0-connected dga's,
using mixed Hodge extensions (see \cite{HT} and \cite{Mo} for similar constructions based on
bigraded extensions).

\begin{teo}\label{minim_ho}
For every 0-connected mixed Hodge diagram $A$ there exists a minimal mixed Hodge dga $M$ together with a ho-morphism $\rho:M\hto A$
which is a quasi-isomorphism.
\end{teo}
\begin{proof}
We will define inductively over $n\geq 1$ and $q\geq 0$ a sequence of mixed Hodge dga's $M(n,q)$ together with ho-morphisms
$\rho(n,q):M(n,q)\hto A$ satisfying the following conditions:
\begin{enumerate}
\item[$(a_{1,0})$] $M(1,0)=\QQ$ has a mixed Hodge structure defined by the trivial filtrations.
\item[$(a_{n,q})$] If $q>0$ then $M(n,q)=M(n,q-1)\otimes_\xi\Lambda(V)$ is a minimal extension of degree $n$.
The morphism $\rho(n,q)^*_\QQ:H^i(M(n,q))\to H^i(A_\QQ)$ is an isomorphism for
all $i\leq n$, and the morphism $i^*:H^n(C(\rho(n,q-1)))\to H^n(C(\rho(n,q)))$ is trivial.
\item[$(a_{n,0})$] If $n>1$ then $M(n,0)$ is the colimit of a sequence $\cdots \subset M(n-1,q)\subset M(n-1,q+1)\subset \cdots$
and the map $\rho(n,0):M(n,0)\hto A$ is the induced ho-morphism.
\end{enumerate}
Then the mixed Hodge dga $M=\cup_n M(n,0)$, together with the induced ho-morphism $\rho:M\hto A$ will be the required quasi-isomorphism.
\\

Assume that we have constructed a minimal mixed Hodge dga $\wt M=M(n,q-1)$ and a ho-morphism $\wt\rho=\rho(n,q-1)\hto A$
satisfying $(a_{n,q-1})$.
Consider the filtered vector spaces of degree $n$ given by
$$V_\QQ=H^{n}(C_1(\wt\rho_\QQ)),\,V_i=H^{n}(C_1(\wt\rho_i)),\text{ for }i\in I\text{ and }V_\CC=H^{n}(C_{1,0}(\wt\rho_\CC)).$$
By Lemma $\ref{mixedcone}$ the mixed cone $C(\wt\rho)$ is a mixed Hodge complex.
Hence we have filtered isomorphisms
$(V_\QQ,W)\otimes \CC\cong (V_i,W)\cong (V_\CC,W)$
making the triple
$(V_\QQ,\Dec W,F)$
into a mixed Hodge structure. By
Lemma $\ref{seccions_compatibles}$ there are sections
$\sigma_\QQ:V_\QQ\to Z^n(C_1(\wt\rho_\QQ))$ and $\sigma_i:V_i\to Z^n(C_1(\wt\rho_i))$
compatible with $W$, together with a section
$\sigma_\CC:V_\CC\to Z^n(C_1(\wt\rho_\CC)),$
compatible with both filtrations $W$ and $F$.
Define filtered dga's
$$M_\QQ=\wt M_\QQ\otimes\Lambda(V_\QQ),\,M_i=\wt M_i\otimes\Lambda(V_i),\text{ for }i\in I\text{ and }M_\CC=\wt M_\CC\otimes\Lambda(V_\CC).$$
The corresponding filtrations are defined by multiplicative extension. The sections $\sigma_\QQ$, $\sigma_i$ and $\sigma_\CC$ allow
to define differentials such that $d(W_pV_\QQ)\subset W_{p-1}\wt M_\QQ^{n+1}$ and $d(F^pV_\CC)\subset F^q\wt M_\CC^{n+1}$, and maps 
$\rho_\QQ:M_\QQ\to A_\QQ$, $\rho_i:M_i\to A_i$ and $\rho_\CC: M_\CC\to A_\CC$ compatible with the corresponding filtrations.
Since by hypothesis $\wt M_\QQ$ is generated in degrees $\leq n$, it follows that $dV_\QQ\subset \wt M^+_\QQ\cdot \wt M^+_\QQ$.
\\

Since $M_i$ is $E_1$-cofibrant, by Theorem $\ref{rcofs}$
for every solid diagram
$$
\xymatrix{
M_i\ar[d]_{\rho_i}\ar@{.>}[r]^{\varphi_u}&{M_j}\ar[d]^{\rho_j}\\
A_i\ar[r]^{\varphi_u}&A_j
}
$$
there exists a
morphism $\varphi_u:M_i\to M_j$ together with a $1$-homotopy $ R_u$ from $\rho_j\varphi_u$ to $\varphi_u\rho_i$.
Since $M_i$ are minimal, $\varphi_u$ is an isomorphism. Hence we can transport the filtrations $F$ of $M_\CC$
to $M_\QQ\otimes\CC$. Let $M(n,q)=M_\QQ$.
The morphisms $\rho_\QQ$ and $\rho_\CC\varphi_u$,
together with the homotopies $R_u$ define a
ho-morphism $\rho:M(n,q)\hto A$ satisfying $(a_{n,q})$. This ends the inductive step.
\end{proof}

To construct
minimal models of morphisms of mixed Hodge diagrams
we follow the same steps as in the classical construction for morphisms of dga's (see e.g. $\S$14 of \cite{FHT}).

\begin{defi}

A \textit{relative minimal mixed Hodge dga} is given by an inclusion $M\hookrightarrow \wt M$ of mixed Hodge dga's
where $\wt M=M\otimes\Lambda(V)$ is a colimit mixed Hodge extensions of $M$
satisfying
$$d(W_pV)\subset W_{p-1} (M^+\otimes\Lambda(V))\oplus W_{p-1}(M\otimes\Lambda^{\geq 2}V).$$
\end{defi}

\begin{teo}\label{modelmorfisme}
For every morphism $f:A\to B$ of 0-connected mixed Hodge diagrams there exists a relative minimal mixed Hodge dga
$\wt f:M\to \wt M$, with $M$ minimal, and a commutative diagram
$$\xymatrix{
A\ar[r]^f&B\\
M\ar@{~>}[u]^{\rho}\ar[r]^{\wt f}&\wt M\ar@{~>}[u]_{\rho'}
}$$
where the vertical ho-morphisms are quasi-isomorphisms.
\end{teo}
\begin{proof}
We work inductively over $n\geq 1$ and $q\geq 0$ as follows.
By Theorem $\ref{minim_ho}$ there is a minimal mixed Hodge dga $M$ and a quasi-isomorphism $\rho:M\hto A$.
As a base case for our induction we take
$\wt M(1,0)=M$ and $\rho(1,0)=f\rho$.
Assume inductively that we have constructed a relative minimal 
mixed Hodge dga $f(n,q-1):M\to \wt M(n,q-1)$ 
together with a quasi-isomorphism $\wt\rho(n,q-1):\wt M(n,q-1)\hto B$.
The inductive step follows as in Theorem $\ref{minim_ho}$, by taking a mixed Hodge extension defined
via the mixed cone of the ho-morphism $\wt\rho(n,q-1)$.
The ho-morphism $$\rho':=\bigcup \wt\rho(n,0):\wt M:=\bigcup_n \wt M(n,0)\hto A$$ together with the inclusion
$\wt f:M\to \wt M$ give the required commutative diagram.\\
\end{proof}

\subsection*{Formality of mixed Hodge diagrams}
Deligne's splitting of mixed Hodge structures defines a 1-splitting of mixed Hodge algebras over $\CC$.
We show that this descends to a 1-splitting over $\QQ$ for minimal mixed Hodge algebras of finite type.

\begin{lem}[cf. \cite{Mo}, Thm. 9.6]\label{formalsC}
Let $f:A\to B$ be a morphism of mixed Hodge dga's of finite type and let $f_\CC:=f\otimes_\QQ\CC$.
Then $\Dec f_\CC$ admits a 0-splitting over $\CC$. Furthermore, if $A$ and $B$ are $E_1$-cofibrant then $f_\CC$ admits a 1-splitting over $\CC$.
\end{lem}
\begin{proof}
Since for all $n\geq 0$ the triple
$(A^n,\Dec W,F)$
is a mixed Hodge structure, by Lemma $\ref{Desplitting}$ we have functorial decompositions
$$A^n_\CC=\bigoplus I^{p,q}_n,\text{ with }\Dec W_mA^n_\CC=\bigoplus_{p+q\leq m} I^{p,q}_n.$$
Since the differentials and products of $A$ are morphisms of mixed Hodge structures, we have
$d(I^{p,q}_n)\subset I^{p,q}_{n+1}$ and $I^{p,q}_n\cdot I^{p',q'}_{n'}\subset I^{p+p',q+q'}_{n+n'}$.
Then 
$A^{p,n-p}:=\bigoplus_r I^{-p-r,r}_{n}$
define a $0$-splitting of the filtered dga $(A_\CC, \Dec W)$.
Apply the same argument to define a 0-splitting for $B_\CC=\bigoplus B^{p,n-p}$.
Since $\Dec f:\Dec A\to \Dec B$ is a morphism of graded mixed Hodge structures 
and Deligne's splittings are functorial, the morphism $\Dec f_\CC$ is
compatible with these 0-splittings.\\

For $E_1$-cofibrant dga's the décalage functor has an inverse defined by shifting the weight filtration. Indeed, if
$A$ is $E_1$-cofibrant, 
by Lemma $\ref{dec_cof}$ we have $W_pA^n=\Dec W_{p-n}A^n$.
Then $\wt A^{p,n-p}:=A^{p-n,2n-p}$
define a 1-splitting of $A$ with respect to the filtration $W$. The same argument applies to $B$. The map
$f_\CC$ is compatible with these 1-splittings.
\end{proof}

\begin{lem}\label{formalsQ}
Let $f:M\hookrightarrow \wt M$ be a relative minimal mixed Hodge dga, with $M$ minimal. If $M$ and $\wt M$ have finite type
then $f$ admits a 1-splitting over $\QQ$.
\end{lem}
\begin{proof}
Let $t_nM$ denote the subalgebra of $M$ generated by 
$M^{\leq n}$. Likewise, denote by $t_n\wt M$ 
the subalgebra of $\wt M$ generated by
$M^{\leq n+1}+\wt M^{\leq n}$.
The minimality conditions on $M$ and $f$ ensure that both $t_nM$ and $t_n\wt M$ are stable by the differentials.
Hence $t_nM$ and $t_n\wt M$ are filtered sub-dga's of $M$ and $\wt M$ respectively.
Denote by $t_nf:t_nM\to t_n\wt M$ the restriction of $f$.
Then $f$ can be written as the inductive limit of $t_{n}f$ over $n\geq 0$.
Since $f$ is a morphism of $E_1$-cofibrant dga's of finite type, it follows that:
\begin{enumerate}[(i)]
 \item $t_{n} f:t_nM\to t_n\wt M$ is a morphism of $E_1$-cofibrant finitely generated dga's.
 \item $t_{n} f$ is stable by the automorphisms of $ f$: there is a map $\Aut(f,W)\to \Aut(t_{n}f,W)$.
 \item \label{limitautos} There is an inverse system of groups $(\Aut(t_{n}f,W))_{n}$
 and an isomorphism of groups 
 $$\Aut( f,W) \lra \invlim \Aut(t_{n} f,W).$$ 
\end{enumerate}
Since $t_{n}(f\otimes\CC) \cong t_{n} f \otimes\CC$, the morphisms $t_{n}f\otimes\CC$
inherit 1-splittings by Lemma $\ref{formalsC}$.
Hence the morphisms $t_{n}f$ admit 1-splittings by (i) and Theorem $\ref{descens_rsplittings}$.
It suffices to show that the 1-splittings of $t_{n}f$ allow to define a 1-splitting of
$f$. This follows as in the proof of Theorem 6.2.1 of \cite{Operads}, using properties (ii) and (iii).
\end{proof}

\begin{defi}
We say that a mixed Hodge diagram $A$ has \textit{homotopy finite type} if there exists a quasi-isomorphism $M\hto A$
where $M$ is a minimal mixed Hodge dga of finite type.
\end{defi}

\begin{teo}\label{formals_MHD}
The rational component of every morphism $f:A\to B$ of 0-connected mixed Hodge diagrams with homotopy finite type is $E_1$-formal.
\end{teo}
\begin{proof}
By Theorem $\ref{modelmorfisme}$ there exists a minimal model $\wt f:M\to \wt M$ of $f$.
By assumption both $M$ and $\wt M$ have finite type.
By Lemma $\ref{formalsC}$ $\wt f_\CC$ admits a 1-splitting.
Therefore $\wt f_\QQ$ admits a 1-splitting by Lemma $\ref{formalsQ}$.
We obtain a commutative diagram
$$
\xymatrix{
\ar[d]^{f_\QQ}(A_\QQ,d,W)&\ar[l]_-\sim(M_\QQ,d,W)\ar[r]^-\cong\ar[d]^{\wt f_\QQ}&(E_1(M_\QQ),d_1,W)\ar[d]^-{E_1(\wt f_\QQ)}\ar[r]^-\sim&(E_1(A_\QQ),d_1,W)\ar[d]^{E_1(f_\QQ)}\\
(B_\QQ,d,W)&\ar[l]_-\sim(\wt M_\QQ,d,W)\ar[r]^-\cong&(E_1(\wt M_\QQ),d_1,W)\ar[r]^-\sim&(E_1(B_\QQ),d_1,W)\\
}
$$
where the horizontal arrows are $E_1$-quasi-isomorphisms.
\end{proof}

The previous result can be restated in terms of a formality property for the forgetful functor
$$U^{ft}_\QQ:\Ho\left(\MHD^{ft}\right)\lra \Ho_1\left(\Fdga{ft}{\QQ}\right)$$
defined by sending every 0-connected mixed Hodge diagram with homotopy finite type to its rational component.
\begin{cor}
There is an isomorphism of functors $E_1\circ U^{ft}_\QQ\cong U_\QQ^{ft}$.
\end{cor}

\section{Mixed Hodge Theory of Complex Algebraic Varieties}

\subsection*{Mixed Hodge diagrams associated with algebraic varieties}
We recall Navarro's functorial construction of mixed Hodge diagrams associated
 with complex algebraic varieties within the context of cohomological descent
categories and the extension criterion of functors of \cite{GN}.

\begin{nada}The Thom-Whitney simple functor defined by Navarro in \cite{Na} for
strict cosimplicial dga's is easily adapted to the cubical setting (see 1.7.3 of \cite{GN}).
Given a non-empty finite set $S$, denote by $L_S$ the dga 
over $\kk$ of smooth differential forms over the hyperplane of the affine space $\mathbb{A}^{S}_\kk$, defined by the equation 
$\sum_{s\in S}t_s=1.$ \\

For $r\geq 0$, let $\sigma[r]$ be the increasing filtration of $L_S$ defined by
letting $t_s$ be of pure weight $0$ and $dt_s$ of pure weight $-r$, for every generator $t_s$ of degree 0 of $L_S$,
and extending multiplicatively.
For every filtered dga $(A,W)$, we have a family of filtered dga's $L_S^r(A)$
indexed by $r\geq 0$, given by:
$$W_pL_S^r(A):=\bigoplus_q \left(\sigma[r]_qL_S\otimes W_{p-q}A\right).$$
\end{nada}

\begin{defi}
The \textit{$r$-Thom-Whitney simple} of 
a codiagram of filtered dga's $A=((A,W)^\alpha)$
is the filtered dga 
$\mathbf{s}^r_{TW}(A,W)$
defined by the end
$$W_p\mathbf{s}^r_{TW}(A)=\int_\alpha \bigoplus_q\left(\sigma[r]_qL_\alpha\otimes W_{p-q}A^\alpha\right).$$
For a codiagram of bifiltered dga's $A=((A,W,F)^\alpha)$, the
\textit{$(r,0)$-Thom-Whitney simple} $\mathbf{s}^{r,0}_{TW}(A,W,F)$ is defined analogously:
$$W_pF^q\mathbf{s}^{r,0}_{TW}(A)=\int_\alpha \bigoplus_{l}\left(\sigma[r]_{l}L_\alpha\otimes W_{p-l}F^qA^\alpha\right).$$
\end{defi}

\begin{defi}
Let $A$ be a cubical codiagram of mixed Hodge diagrams.
The \textit{Thom-Whitney simple of $A$} is the diagram of dga's
$$\mathbf{s}_{TW}(A)=\left(\mathbf{s}_{TW}^{1}(A_\QQ,W)
\stackrel{\mathbf{s}(\varphi)}{\dashleftarrow\dashrightarrow}
\mathbf{s}_{TW}^{1,0}(A_\CC,W,F)\right) 
.$$
\end{defi}

\begin{teo}\label{mhd_es_descens}
The category of mixed Hodge diagrams $\MHD$ with the class of quasi-isomorphisms and the Thom-Whitney 
simple functor $\mathbf{s}_{TW}$
is a cohomological descent category.
\end{teo}
\begin{proof}
The Thom-Whitney simple of a cubical codiagram of mixed Hodge diagrams is a mixed Hodge diagram. Indeed, it
suffices to prove that the associated functor of strict cosimplicial objects is a mixed Hodge diagram.
This follows from 7.11 of \cite{Na}. Consider the functor $U_\QQ:\MHD\lra \dga{}{\QQ}$ 
defined by sending every mixed Hodge diagram $A$ to the dga $A_\QQ$ over $\QQ$.
This functor commutes with the Thom-Whitney simple. The class of quasi-isomorphisms of mixed Hodge diagrams
is obtained by lifting the class of quasi-isomorphisms of dga's.
By Proposition 1.7.4 of \cite{GN} the category of dga's admits a cohomological descent structure.
Hence by Proposition 1.5.12 of loc.cit., this lifts to a cohomological descent structure on $\MHD$.
\end{proof}

Denote by $\Sch{}{\CC}$ the category of complex reduced schemes,
that are separated and of finite type.

\begin{teo}[\cite{Na}, $\S$9]\label{extensionavarro}
There exists a functor $$\HH dg:\Sch{}{\CC}\lra \Ho\left(\MHD\right)$$
satisfying the following conditions:
\begin{enumerate}[(1)]
\item The rational component of $\HH dg(X)$ is $A_\QQ(X)\cong \Aa_{Su}(X^{an};\QQ)$.
\item The cohomology $H(\HH dg(X))$ is the mixed Hodge structure of the cohomology of $X$.
\end{enumerate}
\end{teo}
\begin{proof}
Denote by $\mathbf{V}^2(\CC)$ the category of pairs $(X,U)$ where $X$ is a smooth projective scheme
over $\CC$ and $U$ is an open subscheme of $X$ such that $D=X-U$ is a normal crossings divisor.
By Theorem 8.15 of \cite{Na} there exists a functor $\HH dg:\mathbf{V}^2(\CC)\lra\MHD$ such that:
\begin{enumerate}[(1')]
\item The rational component of $\HH dg(X,U)$ is $A_\QQ(U)\cong \Aa_{Su}(U^{an};\QQ)$.
\item The cohomology $H(\HH dg(X,U)$ is the mixed Hodge structure of the cohomology of $U$.
\end{enumerate}
By Theorem $\ref{mhd_es_descens}$ the Thom-Whitney simple endows the category of mixed Hodge diagrams with a cohomological descent structure.
For every elementary acyclic diagram
$$
\xymatrix{
\ar[d]_g(\wt Y,\wt U\cap \wt Y)\ar[r]^j&(\wt X,\wt U)\ar[d]^{f}\\
(Y,U\cap Y)\ar[r]^i&(X,U)&
}
$$
of $\mathbf{V}^2(\CC)$, the mixed Hodge diagram $\HH dg(X,U)$ is quasi-isomorphic to the Thom-Whitney simple of
the mixed Hodge diagrams associated with the remaining components.
Therefore the functor 
$$\mathbf{V}_{\CC}^2\xra{\HH dg}\MHD\stackrel{\gamma}{\lra}\Ho\left(\MHD\right)$$
satisfies the hypothesis of Theorem 2.3.6 of loc.cit. on the extension of functors.
\end{proof}

\subsection*{Formality} The previous theorem together with the results of Section $\ref{homotopytheory}$
 lead to the main result of this paper on the $E_1$-formality of complex algebraic varieties.

\begin{teo}\label{formalitat_vars}
Let $f:Y\to X$ be a morphism of complex algebraic varieties.
If $X$ and $Y$ are nilpotent spaces then
the rational $E_1$-homotopy type of $f$ is a formal consequence of the first term of spectral sequence associated with the weight filtration:
there exists a diagram
$$
\xymatrix{
\ar[d]^{f_\QQ}(A_\QQ(X),W)&(M_X,W)\ar[d]^{\wt f_\QQ}\ar[l]_-\sim\ar[r]^-\sim&(E_1(A_\QQ(X)),W)\ar[d]^{E_1(f_\QQ)}\\
(A_\QQ(Y),W)&(\wt M_Y,W)\ar[l]_-\sim\ar[r]^-\sim&(E_1(A_\QQ(Y)),W)\\
}
$$
which commutes in the homotopy category $\Ho_1(\Fdga{}{\QQ})$.
\end{teo}
\begin{proof}
By Theorem $\ref{extensionavarro}$ there is a functor
$\HH dg:\Sch{}{\CC}\lra\Ho(\MHD)$ whose rational component is the Sullivan-de Rham functor
$X\mapsto A_\QQ(X)=\Aa_{Su}(X^{an};\QQ)$. In addition, for a nilpotent space $X$, the minimal model of $A_\QQ(X)$ has finite type.
The result follows from Theorem $\ref{formals_MHD}$.
\end{proof}

The previous result can be restated in terms of a formality property for the composite functor
$$\Aa^{nil}_\QQ:\Sch{nil}{\CC}\xra{\HH dg}\Ho\left(\MHD^{ft}\right)\xra{U_\QQ} \Ho_1\left(\Fdga{ft}{\QQ}\right)$$
defined by sending nilpotent complex algebraic varieties to their Sullivan-de Rham algebra endowed with the multiplicative weight filtration.
\begin{cor}
There is an isomorphism of functors
$E_1\circ\Aa_\QQ^{nil} \cong \Aa_\QQ^{nil}$.
\end{cor}

\section{An Application: The Hopf Invariant}
We provide an expression of the Hopf invariant of algebraic morphisms 
$f:\CC^2\setminus\{0\}\to \PP^1_\CC$ in the context of algebraic geometry, using Theorem $\ref{formalitat_vars}$ and intersection theory.
The results of this section can be easily generalized to morphisms
$f:\CC^{n+1}\setminus \{0\}\to \PP_\CC^{n}$, for $n\geq 1$. However, for the sake of simplicity we shall only 
develop the case $n=1$. At the end of the section we compute the Hopf invariant for a particular
 class of morphisms which includes the Hopf fibration.

\subsection*{The Hopf invariant}
We first recall Whitehead's definition of the Hopf invariant in the context
of differential forms and show that it
can be computed in the context of rational homotopy,
via Sullivan minimal models.
\\

Consider a differentiable morphism $f:S^{3}\to S^2$. Denote by
$f^*:\Aa_{dR}(S^2)\lra \Aa_{dR}(S^3)$ the induced morphism of algebras.
Choose fundamental classes $[S^2]$ and $[S^3]$ together with normalized volume forms $w_2$ and $w_3$ of $S^2$ and $S^3$ respectively,
satisfying
$\int_{S^2}w_2=\int_{S^3} w_3=1.$
Let $\theta$ be a one-form in $\Aa_{dR}(S^2)$ such that $f^*(w_2)=d\theta$.
The \textit{Hopf invariant of $f$} is defined by
$$H(f):=\int_{S^{3}}\theta\wedge d\theta.$$
While the definition of $H(f)$ is independent of the choice of $\theta$ and the orientation of $S^2$,
it does depend on the choice of orientation of $S^3$.
Homotopic maps have the same Hopf invariant.
Geometrically, $H(f)$ is given by the linking number of pre-images of two distinct regular values of $f$.
In particular it is always an integer number, and it defines a homomorphism
$H:\pi_3(S^2)\lra \ZZ$ (see e.g. $\S$18 of \cite{BottTu}).

\begin{nada}\label{modelsesferes}
We define the normalized minimal model of a continuous map $f:S^{3}\to S^2$ as follows.
Let $M(S^2)=\Lambda(\alpha,\beta)$ be the free $\QQ$-dga generated by $\alpha$ in degree 2 and $\beta$ in degree 3 with
differentials $d\alpha=0$ and $d\beta=\alpha^2$.
The morphism $\rho_2:M(S^2)\to \Aa_\QQ(S^2)$ defined by sending $\alpha$ to the volume form $w_2$ of $S^2$ is a Sullivan minimal model of $S^2$.
Likewise, let $M(S^3)=\Lambda(\gamma)$ be the free $\QQ$-dga generated by $\gamma$ in degree 3 with trivial differential.
The morphism
$\rho_3:M(S^3)\to \Aa_\QQ(S^3)$ defined by sending $\gamma$ to the volume form $w_3$ of $S^3$ is a Sullivan minimal model of $S^3$.
By dimensional arguments, every possible morphism $\wt f_\lambda:M(S^2)\to M(S^3)$ is of the form $\alpha\mapsto 0$ and 
$\beta\mapsto \lambda\cdot \gamma$, with $\lambda\in\QQ$.
Furthermore, any two such homotopic morphisms coincide.
The map $\wt f_\lambda$ is a minimal model of $f^*$ if and only if the diagram
$$
\xymatrix{\ar@{}[dr] |{\simeq}
\Aa_\QQ(S^2)\ar[r]^{f^*}&\Aa_\QQ(S^3)\\
M(S^2)\ar[u]^{\rho_2}\ar[r]^{\widetilde f_\lambda}&M(S^3)\ar[u]_{\rho_3}
}
$$
commutes up to homotopy.
In such case we say that
the above diagram is a \textit{normalized minimal model} of $f^*$ with respect to the chosen volume forms.
Note that as in the definition of the Hopf invariant, the sign of $\lambda$ depends on the choice of orientation of $S^3$.
\end{nada}

\begin{prop}\label{hopf}
Let $f:S^3\to S^2$ be a differentiable morphism. Then $\wt f_\lambda$
is a normalized minimal model of $f^*$ if and only if $H(f)=\lambda$.
\end{prop}
\begin{proof}
Let $\theta$ be a one-form of $\Aa_\QQ(S^3)$ satisfying $d\theta=f^*(w_2)$.
Define $h:M(S^2)\to \Aa_\QQ(S^3)\otimes\Lambda(t,dt)$ by letting
$h(\alpha)=d(\theta\cdot t)$ and $h(\beta)=\lambda w_3(1-t^2).$
Then $\delta^0h=\rho_3\circ \wt f_\lambda$ and $\delta^1 h=f^* \circ\rho_2$.
For $h$ to be a morphism of dga's it is necessary and sufficient that $h(\alpha)^2=dh(\beta)$.
This is the case only when $d\theta \cdot \theta=\lambda w_3$.
We have
$H(f)=\int_{S^3} d\theta\cdot \theta=\int_{S^3}\lambda w_3=\lambda.$
\end{proof}

\subsection*{Weight spectral sequence}
We study the rational homotopy type and the Hopf invariant of certain algebraic morphisms
of complex algebraic varieties, via the weight filtration.
\begin{defi}
Let $f:\CC^{2}\setminus\{0\}\to \PP^{1}_\CC$ be a morphism of complex algebraic varieties,
and let $i:S^{3}\hookrightarrow \CC^{2}\setminus\{0\}$ denote the inclusion.
We call $H(f\circ i)$ the \textit{Hopf invariant of $f$}.
\end{defi}
Consider the smooth compactification
$U:=\CC^{2}\setminus\{0\}\hookrightarrow X:=\widetilde{\PP}^{2}$
into the blown-up complex projective plane at the origin. We have a diagram
$$
\xymatrix{
\ar[d]\PP^1_\infty\ar[r]^-{i}&\widetilde{\PP}^{2}\ar[d]^\pi&\PP^1_E\ar[l]_-{j}\ar[d]\\
\PP^1_\infty\ar[r]&\PP^{2}&\{0\}\ar[l]
}
$$
where $\PP^{1}_\infty$ and $\PP^{1}_E$ are complex projective lines denoting the hyperplane at infinity and the exceptional divisor respectively.
The cohomology ring of $\wt\PP^{2}$ is given by
$$H^*(\widetilde{\PP}^{2};\QQ)=\QQ\langle a,b\rangle,\text{ with }a\cdot b=0\text{ and }a^{2}=- b^{2},$$
where $a=i_* 1_\infty$ and $b=j_*1_E$ denote the classes of $\PP^{1}_\infty$ and $\PP^{1}_E$ respectively.
The cohomology ring of the complement $D:=\widetilde{\PP}^{2}_\CC-U=\PP^{1}_\infty\sqcup \PP^1_E$ can be written as
$$H^*(D;\QQ)=\QQ\langle x,y\rangle,\text{ with }x\cdot y=0,\,x^{2}=0\text{ and }y^{2}=0,$$
where $x$ and $y$ denote the classes of a point in $\PP^{1}_\infty$ and $\PP^{1}_E$ respectively.\\

The differentials and products of the weight
spectral sequence can be computed in the Chow rings, using intersection theory.
We will use the following result (see Proposition 2.6 of \cite{Fulton}).
\begin{prop}\label{cherns}
Let $j:D\to X$ denote the inclusion of a Cartier divisor $D$ on a scheme $X$.
\begin{enumerate}[(a)]
\item If $\alpha$ is a cycle on $X$ then $j_*j^*\alpha=c_1(\Oo_X(D))\cap \alpha$.
\item If $\alpha$ is a cycle on $D$ then $j^*j_*\alpha=c_1(N_D)\cap \alpha$, where $N_D=j^*(\Oo_X(D))$.
\end{enumerate}
\end{prop}

The first Chern classes associated with the morphisms $i$ and $j$ above are given by
$$c_1(\Oo_X(\PP^1_\infty))=a,\,c_1(\Oo_X(\PP^1_E))=b,\,c_1(N_\infty)=x\text{ and }c_1(N_E)=-y.$$
Using Proposition $\ref{cherns}$ we obtain the following intersection products:
$$
1_\infty\cdot a:=i^*a=i^*i_*1_\infty=c_1(N_\infty)=x\text{ and }
1_E\cdot b:=j^*b=j^*j_*1_E=c_1(N_E)=-y.$$
Since $\PP^1_\infty\cap \PP^1_E=\emptyset$, it follows that $1_\infty\cdot b=0$ and $1_E\cdot a=0$.\\

With these results we can write the first term of the weight spectral sequence associated with the compactification $U\hookrightarrow \widetilde{\PP}^{2}$.
By definition the only non-trivial terms are
$$E_1^{0,q}(U)=H^q(\widetilde{\PP}^{2};\QQ)\text{ and }E_1^{-1,q}(U)=H^{q-2}(D;\QQ),$$
and the differentials $d_1:H^{q-2}(D)\lra H^q(\widetilde{\PP}^{2})$ are given by the Gysin morphisms $i_*$ and $j_*$.
Let $u=1_\infty$ and $v=1_E$. Since $x=u\cdot a$ and $y=-v\cdot b$
we can write
$$E_1^{*,*}(U)=\Lambda\left(u,v,a,b\right)/R$$
as the quotient of the free bigraded algebra generated by $u$ and $v$ of bidegree (-1,2), and $a$ and $b$ of bidegree (0,2), by the ideal of relations 
$$R=\left(uv,\,ub,\,va,\,a^{2}+b^{2},\,ab\right).$$
The differential is defined on the generators by $du=a$, $dv=b$ and $da=db=0$.\\

Since $\PP^1_\CC$ is smooth and compact the first term of the associated weight spectral sequence satisfies
$$E_1^{0,q}(\PP^{1}_\CC)=H^{q}(\PP_\CC^{1};\QQ)\text{ and }E_1^{-p,q}(\PP^{1}_\CC)=0\text{ for }p\neq 0.$$
Therefore we can write
$$E_1^{*,*}(\PP^1_\CC)=\Lambda(\alpha) / \alpha^{2}$$
as the quotient of the free bigraded algebra generated by $\alpha$ in bidegree $(0,2)$ by the ideal $(\alpha^{2})$, and with trivial differential.

\begin{prop}Let $f:U:=\CC^{2}\setminus\{0\}\to \PP^1_\CC$ be a morphism of complex algebraic varieties extending to a morphism
$g:\wt\PP^2_\CC\to \PP^1_\CC$.
\begin{enumerate}[(1)]
\item There exists a unique $\eps \in \ZZ$ such that $E_1(g):E_1(\PP^1_\CC)\lra E_1(U)$ satisfies $\alpha\mapsto \eps (a\pm b)$.
\item The map $\wt f:M(S^2)\to M(S^3)$ given by $\beta\mapsto \eps^2 \gamma$ defines a normalized minimal model of $f^*$.
\item The Hopf invariant of $f$ is $H(f)=\eps^2$.
\end{enumerate}
\end{prop}
\begin{proof}
By dimensional reasons any morphism $E_1(g):E_1(\PP^1_\CC)\to E_1(U)$ is of the form $\alpha\mapsto \eps_1 a+\eps_2 b$, where $\eps_i\in \QQ$.
The compatibility condition for $\alpha^2=0$ implies that $\eps_1=\pm \eps_2$.
The weight spectral sequence associated with a compactification is defined over $\ZZ$. Hence, $E_1(g)$ is defined  over $\ZZ$, and (1) follows.
Let us prove (2).
Define a quasi-isomorphism $\rho:M(S^2)\to E_1(\PP^1_\CC)$ by letting $\rho(\alpha)=\alpha$ and $\rho(\beta)=0$.
Likewise, define a quasi-isomorphism $\rho':M(S^3)\to E_1(U)$ by $\rho'(\gamma)=ua+vb$.
The morphism $h:M(S^2)\to E_1(U)\otimes\Lambda(t,dt)$ defined by 
$h(\alpha)=\eps(a\pm b)t-\eps(u\pm v)dt$ and $h(\beta)=\eps^2\cdot(ua+vb)(1-t^{2})$
is a homotopy from $\rho'\circ \wt{f}$ to $E_1(g)\circ\rho$. Hence $\wt f$ is a minimal model of $E_1(g)$.
By Theorem $\ref{formalitat_vars}$ this defines a normalized minimal model of $f^*$.
This proves (2). Assertion (3) follows from Proposition $\ref{hopf}$.
\end{proof}

\begin{example}
Let $q\geq 1$, and let $f:U\to \PP^1_\CC$ be the morphism defined by
$(x_0,x_1)\mapsto [x_0^q:x_1^q]$. For $q=1$, this morphism is the Hopf fibration.
Then $f$ extends to a morphism $g:\wt\PP^2_\CC\to \PP^1_\CC$.
The induced morphism at the level of spectral sequences
$E_1(g):E_1(\PP^1_\CC)\lra E_1(U)$ is given by $\alpha\mapsto q(a-b)$. Indeed,
the pre-image $g^{-1}(p)$ of a point $p$ is a family of $q$ lines which intersect both $\PP^1_E$ and $\PP^1_\infty$ at a point in $\wt\PP^2_\CC$.
We find that $H(f)=q^2$.
\end{example}

\subsection*{Acknowledgments}
We want to thank V. Navarro for his valuable comments and suggestions.

\linespread{1}
\bibliographystyle{amsalpha}
\bibliography{bib}
\mbox{}\\
\linespread{1.2}

\end{document}